\documentclass[12pt]{article}
\usepackage[titletoc]{appendix}
\usepackage[labelsep=period]{caption}
\usepackage{bm}
\usepackage{fullpage}
\usepackage{epsfig}
\usepackage{graphics}
\usepackage{latexsym}
\usepackage{amsmath}
\usepackage{amsfonts}
\usepackage{amssymb}
\usepackage{mathrsfs}
\usepackage{pifont}
\usepackage{yhmath}
\usepackage{bbm}
\usepackage{dsfont}
\usepackage{eufrak}
\usepackage{wasysym}
\usepackage{underscore}
\usepackage{epstopdf}
\usepackage{fontenc}
\usepackage{amsthm}
\usepackage{url}
\newtheorem{theorem}{Theorem}[section]

\newtheorem{corollary}[theorem]{Corollary}

\newtheorem{lemma}[theorem]{Lemma}

\newtheorem{conjecture}{Conjecture}[section]

\newcommand{\D}{\Delta}

\newcommand{\phibar}{\overline{\varphi}}

\def\qed{\hfill \rule{4pt}{7pt}}
\usepackage{fullpage,graphicx}
\DeclareMathAlphabet{\mathpzc}{OT1}{pzc}{m}{it}
\title{\bf A note on Goldberg's conjecture on total chromatic numbers}
\author{\quad {\small Yan Cao$^{a}$}
	\quad 	{\small Guantao Chen$^{b}$\thanks{Supported in part by NSF grant DMS-1855716}} 
	\quad {\small	Guangming Jing$^{c}$\thanks{Supported in part by NSF grant DMS-2001130}},
	\unskip\\[.5em]
	{\small $^a$ Department of Mathematics, West Virginia University,
		Morgantown, WV 26505, USA} \smallskip\\
	{\small $^b$ Department of Mathematics and Statistics, Georgia State University
		Atlanta, GA 30303, USA} \smallskip\\
	{\small $^c$ Department of Mathematics, Augusta University, Augusta, GA 30912, USA}
}
\begin{document}
\date{}
\maketitle

\begin{abstract}
Let $G=(V(G), E(G))$ be a multigraph with maximum degree $\D(G)$, chromatic index $\chi'(G)$ and total chromatic number $\chi''(G)$. The Total Coloring conjecture proposed by Behzad and Vizing, independently,  states that $\chi''(G)\leq \D(G)+\mu(G) +1$ for a multigraph $G$, where $\mu(G)$ is the multiplicity of $G$.   Moreover, 
 Goldberg conjectured that $\chi''(G)=\chi'(G)$ if $\chi'(G)\geq \Delta(G)+3$ and noticed  
the conjecture holds when $G$ is an edge-chromatic critical graph.  By assuming the Goldberg-Seymour conjecture, we show that   $\chi''(G)=\chi'(G)$  if $\chi'(G)\geq \max\{ \Delta(G)+2, |V(G)|+1\}$ in this note.  Consequently, $\chi''(G) = \chi'(G)$ if $\chi'(G) \ge \D(G) +2$ and $G$ has a spanning edge-chromatic critical subgraph. 
 
 \end{abstract}

\section{Introduction}
Graphs in this paper may contain multiple edges but no loops. We will generally follow
 Stiebits et al. in~\cite{StiebSTF-Book} for notation and terminology. 
Let  $G=(V(G),E(G))$ be a graph with vertex set $V(G)$ and edge set $E(G)$. Denote by $\D(G)$ the maximum degree of $G$.   
 A {\it vertex $k$-coloring} of $G$  is an assignment of $k$ colors to the vertices of $G$ such that no two adjacent vertices receive the same color. The {\it chromatic number} of $G$, denoted by $\chi(G)$, is the minimum $k\geq 0$ such that $G$ admits a vertex $k$-coloring. A {\it $k$-edge-coloring} of $G$ is an assignment of $k$ colors to the edges of $G$ such that no two incident edges receive the same color. The {\it chromatic index} of $G$, denoted by $\chi'(G)$, is the minimum $k\geq 0$ such that $G$ admits a $k$-edge-coloring. A {\it total $k$-coloring} is an assignment of $k$ colors to the vertices and edges of $G$ such that no two adjacent or incident elements of $V(G)\cup E(G)$ receive the same color. The {\it total chromatic number}, denoted by $\chi''(G)$,  is the minimum $k\geq 0$ such that $G$ admits a total $k$-coloring.  Behzad (1965)~\cite{Behzad} and Vizing (1968)~\cite{vizing68}, independently,  conjectured that $\chi''(G)\leq \Delta(G)+ \mu(G) +1$ for any graph $G$. Note that $\chi'(G) \le \D(G) + \mu(G)$.  Thus a stronger version of the conjecture states that $\chi''(G) \le \chi'(G) +1$ for any graph $G$. Clearly, $\chi'(G)\geq\Delta(G)$ and 
 $\chi''(G) \ge \chi'(G)$.  There are a huge number of graphs with $\chi''(G)>\chi'(G)$. For example, all simple graphs with $\chi'(G)=\Delta(G)$ have $\chi''(G)>\chi'(G)$. A graph $G$ is called  {\em  edge-chromatic  critical } if $\chi'(H) < \chi'(G)$ for every proper subgraph $H$. There are edge-chromatic critical graphs with $\chi''(G) > \chi'(G)$. On the other hand, Goldberg~\cite{goldberg84} observed that if the Goldberg-Seymour conjecture is true, then $\chi'(G)=\chi''(G)$ for every edge-chromatic critical graph with $\chi'(G)\geq \Delta(G)+2$. Furthermore, he made the following conjecture.
  
 \begin{conjecture}[Goldberg, 1984]
 	If $\chi'(G) \ge \D(G) +3$, then $\chi''(G) = \chi'(G)$. 
 \end{conjecture}
 It is worth mentioning that the above conjecture does not hold for all edge-chromatic critical graphs. In fact all odd cycles with length not a multiple of $3$ are edge chromatic critical and $\chi''(G)=\chi'(G)+1$. 
 As we know so far there is no progress toward the above Goldberg's conjecture. In this paper, we prove the following result and hope it will shed some light on attacking this conjecture. 
\begin{theorem}\label{thm-main}
 If $G$ is a graph satisfying   $\chi'(G)\geq \max\{\Delta(G)+2, |V(G)|+1\}$, then $\chi''(G)=\chi'(G)$.
\end{theorem}

\section{Preliminaries}
Let $G$ be a graph and $\varphi$ be an edge $k$-coloring of  $G$. 
Since each color class is a matching of $G$,  we have $|E(H)|\leq k \lfloor|V(H)|/2\rfloor$ for any $H\subseteq G$. 
 Therefore for an arbitrary graph $G$, apart from the maximum degree there is another trivial lower bound for the chromatic index: 
  \[
  \chi'(G)\geq \rho(G)=\max\left\{  \frac{ 2|E(H)| }{|V(H)|-1 } : H \subseteq G, |V(H)|\geq3~odd\right\}. 
 \] 
 We call $\rho(G)$  the {\em density} of $G$.
 In the 1970s, Goldberg~\cite{Goldberg}, Gupta~\cite{10.1007/BFb0070378}, and Seymour~\cite{Seymour} independently conjectured that 
for any graph $G$, if $\chi'(G)>\Delta(G)+1$ then $\chi'(G)=\lceil \rho(G)\rceil$. This conjecture is known as the Goldberg-Seymoure conjecture.
In joint work with Wenan Zang, two authors of this paper, Guantao Chen and Guangming Jing~\cite{chen2019proof} gave a proof of the Goldberg-Seymour conjecture.  We assume that the Goldberg-Seymour conjecture is true in this paper.  

Let $G$ be a graph with $\chi'(G)=k$ and $\varphi$ be a $k$-edge-coloring of $G$.  We in this paper always assume that the set of $k$ colors is $[k] = \{1, 2, \dots, k\}$.   For a vertex $v\in V(G)$, denote by $\varphi(v)$ the set of colors assigned to edges incident with $v$ and 
$\phibar(v)$ the sets of colors not-assigned  to edges incident with $v$, i.e., $\phibar(v) = [k] -\varphi(v)$.  Let $W$ be a vertex set of $G$ and  $\partial(W)$ be the set of {\em boundary edges} of $W$, i.e., edges with exact one end in $W$. Let  $\phibar(W)=\cup_{v\in W}\phibar(v)$ and $\varphi(\partial(W))=\cup_{e\in\partial(W)}\varphi(e)$.
We call $W$ {\em elementary} if $\phibar(u)\cap\phibar(v)=\emptyset$ for any distinct vertices $u,v\in W$, and {\em closed} if $\phibar(W) \cap \varphi(\partial(W)) = \emptyset$.   Moreover, 
we call a closed set $W$ {\em strongly closed} if additionally no two edges in $\partial(W)$ assigned the same color.   For a subgraph $H$ of $G$, we call $H$ {\em elementary} and {\em closed} ({\em strongly closed}) if $V(H)$ is elementary and closed (strongly closed),   and denote $\phibar(V(H))$ and $\partial(V(H))$ by $\phibar(H)$ and $\partial(H)$, respectively.  
In terms of critical subgraphs, the Goldberg-Seymour conjecture can be stated as follows: 
\begin{theorem}\label{thm-CJZ}\cite{StiebSTF-Book}
If $G$ is a critical graph with $\chi'(G) \ge \D(G) +2$, then for any edge $e$ there is a $(\chi'(G)-1)$-edge-coloring $\varphi$ such that $G-e$ is elementary. 
\end{theorem}

As a consequence of  Theorems~\ref{thm-main} and \ref{thm-CJZ}, we obtain the following result. 
\begin{corollary}\label{cor-main}
Let $G$ be a graph with $\chi'(G) \ge \D(G) +2$ and $H$ be  a critical edge-chromatic subgraph 
  with $\chi'(H) = \chi'(G)$. If $|V(H)| \ge (|V(G)| -2)/(\chi'(G) -\D(G) -1)$, then $\chi''(G) = \chi'(G)$. 
  \end{corollary}  
\proof Let $G$ and $H$ be defined as the above, $e\in E(H)$ and $\varphi$ be an edge $(\chi'(G)-1)$-coloring of $H-e$. By Theorem~\ref{thm-CJZ}, $H-e$ is elementary under $\varphi$. Thus $\chi'(G)-1\geq\sum_{v\in V(H)}|\phibar(v)|$. Note that we have $|\phibar(v)|=\chi'(G)-1-(d_{H}(v)-1)$ if $v$ is an end-vertex of the uncolored edge $e$, and $|\phibar(v)|=\chi'(G)-1-d_{H}(v)$ otherwise.  Hence we have the following inequality:

\begin{equation*}
\begin{split}
&\chi'(G) -1 \\
\geq&2+\sum_{v\in V(H)}(\chi'(G)-1-d_H(v))\\
\geq&(\chi'(G) -\D(G) -1) |V(H)| +2 \\
\geq&|V(G)|.
\end{split}
\end{equation*} 

By Theorem~\ref{thm-main}, we have $\chi''(G) = \chi'(G)$.  \qed

As a result of Corollary~\ref{cor-main}, we have $\chi''(G) = \chi'(G)$ if $\chi'(G) \ge \D(G) +2$ and $G$ has a spanning critical subgraph.

Given a graph $G$,  a subgraph $H$ is called {\em $k$-dense} if $|V(H)|$ is odd, $|V(H)|\geq 3$ and $2|E(H)|=k(|V(H)|-1)$.   Moreover, 
 $H$ is called {\em maximal $k$-dense} 
 if there does not exist another $k$-dense subgraph $H'$ containing $H$ as a proper subgraph.  

\begin{lemma}\label{elementary}
Let $G$ be a graph, $\varphi$ be a $\chi'(G)$-edge-coloring of $G$ and $H$ be a $\chi'(G)$-dense subgraph of $G$. Then $H$ is elementary and strongly closed under $\varphi$.	
\end{lemma}	

\begin{proof}
Let $G$, $H$ and $\varphi$ be defined as in the lemma, and let $k = \chi'(G)$. Since $|V(H)|$ is odd,  each color class contains at most   $\frac12 (|V(H)|-1)$ edges. 
 Note that $2|E(H)|= k(|V(H)|-1)$ as $H$ is $k$-dense. Since there are $k$ colors, each color class within $H$ has exactly $\frac12 (|V(H)|-1)$ edges. Therefore, for each color $\alpha\in\phibar(V(H))$, there exists exactly one vertex $v\in V(H)$ such that $\alpha\in\phibar(v)$, and for each color $\beta\notin\phibar(V(H))$, it is used for exactly one edge in $\partial(V(H))$. Consequently, under the coloring $\varphi$, $H$ is elementary since $\phibar(u)\cap\phibar(v)=\emptyset$ for two different vertices $u,v\in V(H)$, closed since no color in $\phibar(V(H))$ appears on the boundary edges, and strongly closed since no color appears on two distinct boundary edges.
\end{proof}
In our proof we only use the elementary property in Lemma~\ref{elementary}. We hope that in the future by using the strongly closed property, one can extend the total coloring of some $\chi'(G)$-dense subgraph $H$ of $G$ to $G$ as all the edges connecting $H$ with the rest of $G$ are colored differently with colors in $[k]-\phibar(V(H))$.

\section{Proof of Theorem~\ref{thm-main}}
Let $G$ be a graph with  $\chi'(G)\geq \max\{\Delta(G)+2, |V(G)|+1\}$ and $\chi'(G)=k$. 
The proof of Theorem~\ref{thm-main} is divided into three lemmas in this section. By Lemma~\ref{3.3}, we see that $G$ is a subgraph of a $k$-dense graph $G'$ such that $\chi'(G)=\chi'(G') = k$ and $\D(G') \le k-1$. Then by Lemma~\ref{3.1}, we observe that $\chi''(G') =k$, so Theorem~\ref{thm-main} holds. Lemma~\ref{density1} is used to prove Lemma~\ref{3.3}.

\begin{lemma}\label{3.1}
Let $G$ be a $k$-dense graph. If $\chi'(G) = k$ and $\D(G) \le k-1$, then $\chi''(G) = k$.  
\end{lemma}
\begin{proof}
Let $\varphi$ be an edge $k$-coloring of $G$. Since $G$ is $k$-dense, by Lemma~\ref{elementary}, $G$ is elementary under $\varphi$. For each vertex $v$, since $d(v) \le \D(G) < k$, $\phibar(v) \ne \emptyset$. We then assign $v$ a color from $\phibar(v)$. Since $G$ is elementary under $\varphi$, colors assigned to different vertices are distinct. This vertex coloring plus the original edge coloring $\varphi$ gives a total $k$-coloring of $G$. \end{proof}

\begin{lemma}\label{density1}
Let $G$ be a graph with $\chi'(G)=k\geq \Delta(G)+1$ and $H_1$, $H_2$ be two distinct subgraphs of $G$. If both $H_1$ and $H_2$ are  maximal $k$-dense,  then $V(H_1)\cap V(H_2)=\emptyset$. 
\end{lemma}
\begin{proof}
Assume on the contrary that $V(H_1)\cap V(H_2)\neq\emptyset$. For each $i=1,2$, since $H_i$ is $k$-dense, we have $|E(H_i)|=k(|V(H_i)|-1)/2$ and $|V(H_i)|$ is odd. Let  $H=H_1\cap H_2$ and $H^*=H_1\cup H_2$.  By the maximality of $H_1$ and $H_2$, we have $H_1 - H_2 \ne \emptyset \ne H_2 - H_1$, which in turn gives that $H_1\subsetneq H^*$ and $H_2\subsetneq H^*$.  We consider two cases according to the parity of $|V(H)|$. 

\flushleft{\bf Case 1:} $|V(H)|$ is odd. 

Since $E(H^*) = E(H_1)\cup E(H_2)$ and $E(H) = E(H_1)\cap E(H_2)$, the following equality holds. 
\begin{equation}\label{eq1}
|E(H^*)|=|E(H_1)|+|E(H_2)|-|E(H)|= k (|V(H_1)|+|V(H_2)|-2)/2-|E(H)|
\end{equation}
On the other hand, since both $H_1$ and $H_2$ are maximal $k$-dense,  $H^*$ is not $k$-dense, and so the following holds. 
\begin{equation}\label{eq2}
|E(H^*)|<k(|V(H^*)|-1)/2=k(|V(H_1)|+|V(H_2)|-|V(H)|-1)/2
\end{equation}

Combining \ref{eq1} and~\ref{eq2}, we have $|E(H)|> k(|V(H)|-1)/2$. Hence $\chi'(G)\geq\rho(G)\geq\rho(H)>k$, giving a contradiction.

\flushleft{\bf Case 2:} $|V(H)|$ is even. 

Let $H_1^*=H_1-V(H)$ and $H_2^*=H_2-V(H)$. Clearly, both $H_1^*$ and $H_2^*$ have odd number of vertices.   Since both $H_1^*$ and $H_2^*$ are edge $k$-colorable, the following two inequalities hold. 
\begin{equation}\label{eq3}
\begin{split}
&|E(H_1^*)|\leq k (|V(H_1)|-|V(H)|-1)/2\\
&|E(H_1^*)|\leq k (|V(H_2)|-|V(H)|-1)/2.
\end{split}
\end{equation}
Since both $H_1$ and $H_2$ are $k$-dense, we have the following equalities. 
\begin{equation}\label{eq4}
	\begin{split}
	& k(|V(H_1)|-1)/2=|E(H_1)|=|E(H)|+|E(H_1^*)|+|E(H_1^*,H)|\\
	&k (|V(H_2)|-1)/2=|E(H_2)|=|E(H)|+|E(H_2^*)|+|E(H_2^*,H)|
	\end{split}
\end{equation} 
where $E(X,Y)$ denotes the set of edges between $X$ and $Y$.  
The combination of ~\ref{eq3} and~\ref{eq4} gives the following. 
\begin{equation}\label{eq5}
\begin{split}
& |E(H_1^*,H)|+|E(H)|\geq k\cdot  |V(H)|/ 2\\
&|E(H_2^*,H)|+|E(H)|\geq k\cdot |V(H)|/ 2
\end{split}
\end{equation} 
Therefore $\sum_{x\in V(H)}d_G(x)=|E(H_1^*,H)|+|E(H_2^*,H)|+2|E(H)|\geq k|V(H)|$, we have reached a contradiction to $\Delta(G)<k$.
\end{proof}

\begin{lemma}\label{3.3}

If  $G$ is a graph with $\chi'(G)=k \geq \max\{\Delta(G)+2, |V(G)|+1\}$,  then there exists a $k$-dense graph $G'$ containing $G$ as a subgraph such that  $\chi'(G')=k$ and  $\Delta(G')\leq k-1$.
\end{lemma}

\begin{proof}
Let $G$ and $k$ be defined as stated in the lemma. Notice that  $k = \chi'(G)  \ge \rho(G)$. 	Let $|V(G)|=n$.
If $n$ is even, we add an isolated vertex to $G$. So we may assume  that $n\leq k$ and $n$ is odd. Let $G'$ be a spanning supgraph of $G$ with maximum number of edges such that 
$\chi'(G') = k$ and $\D(G') \le k-1$.  Since each color class is a matching of $G$, we have $\rho(G')\leq k$ and $|E(G')| \leq k(n -1)/2$.  In the remainder of the proof, we will show that $G'$ itself is $k$-dense, i.e., $|E(G')| =k(n-1)/2$.  Let $U = \{ v\in V(G') \ : \ d_{G'}(v) < k-1\}$. We have the following claim for $U$.

\flushleft{\bf Claim 1.}  If $|U|\geq 2$, then there is a $k$-dense subgraph $H$ of $G'$ such that $U\subseteq V(H)$.

Let $u,v$ be two vertices in $U$. We then  add a new edge $e$ between $u$ and $v$ to $G'$ and denote the resulted graph by $G''$.  Note that $\D(G'') \le k-1$. Since $|E(G')|$ is maximum under the conditions that $\chi'(G') = k$ and $\D(G') \le k-1$,  we must have $\chi'(G'') \ge k+1$. By the Goldberg-Seymour conjecture, we have $\lceil\rho(G'')\rceil=k+1$, and therefore $G''$ contains a subgraph $H$ with $2|E(H)| > k(|V(H)| -1) $ and $|V(H)|$ odd.  Since $E(H) - E(G') \subseteq\{e\}$ and $\rho(G') \le k < \rho(H)$, we have $e\in E(H)$ and $2|E(H-e)| = k(|V(H)| -1)$. So, $H-e$ is a $k$-dense subgraph of $G'$ containing both $u$ and $v$. Since $u$ and $v$ are any two vertices in $U$, by Lemma~\ref{density1}, there must exists a unique maximal $k$-dense subgraph of $G'$ containing all vertices of $U$. Thus we have as claimed.

By Claim 1, we then have the following four cases according to the size of $|U|$: $U = \emptyset$; $U = \{u\}$ is a singleton and $d_{G'}(u) = k-2$;  $U =\{u\}$ is a singleton and $d_{G'}(u) \leq k-3$; and there is a $k$-dense subgraph $H$ of $G'$ such that $U\subseteq V(H)$. 


 \flushleft{\bf Case 1.}  $U = \emptyset$. 

In this case, we have $d_{G'}(v)=k-1$ for every $v\in V(G')$.
Since $n \leq k$, we have $\chi'(G')\geq \rho(G') \geq  \frac{2|E(G')|}{|V(G')|-1}=\frac{(k-1)n }{n -1}\geq \frac{(k-1)k}{k-1}=k$ and the equality holds only when $n =k$. Since $\chi'(G)=k$, we have that $G'$ is $k$-dense.
 
 \flushleft {\bf Case 2.} $U = \{u\}$ is a singleton and $d_{G'}(u) = k-2$. 
 
 In this case, we have we have $d_{G'}(v)=k-1$ for every $v\in V(G')$ with $v\neq u$. By counting the total degree, we have the following equality:
 \[
 2|E(G')|=\sum_{v\in V(G')}d_{G'}(v)=(k-1)(n-1)+k-2=k(n-1)+k-1-n.
 \]
  If $n<k$, then we have $2|E(G')|\geq k(n-1)$. Thus $G'$ is $k$-dense since $|E(G')| \leq k(n -1)/2$. If $n =k$, then $k$ is odd and the total degree $\sum_{v\in V(G')}d_{G'}(v)=(k-1)(n-1)+k-2$ is odd, which contradicts the total degree being an even number.



\flushleft{\bf Case 3.}  There is a $k$-dense subgraph $H$ of $G'$ such that $U\subseteq V(H)$. 
		
In this case we may assume $H$ is a maximal $k$-dense subgraph of $G'$. If $H=G'$, then $G'$ is $k$-dense and we are done. Thus we may assume that $V(G)-V(H)\neq\emptyset$.
Let $F$ be the subgraph of $G'$ induced by $V(G) - V(H)$.  In this case, every vertex outside of $H$ has degree $k-1$ and $|E(H)| = k (|V(H)| -1)/2$. Counting edges in $H$ and $F$, and edges between $H$ and $F$, we have the following inequalities: 
\begin{equation*}
		\begin{split}
		&k(|V(H)|-1)+(k-1)|V(F)|+|E(H,F)|\\
		=&2|E(H)|+2|E(H,G'-H)|+2|E(F)|=2|E(G')|\\
		<&2\cdot k (|V(G')|-1)/2 
		=  k(|V(G')|-1) 
		= k(|V(H)| -1) + k |V(F)|, 
		\end{split}
		\end{equation*} 
where $|E(H,F)|$ denotes the number of edges between $H$ and $F$ in $G$.		
Hence, $|V(F)|>|E(H, F)|$. We also notice that 
 since all vertices in $F$ have degree $k-1>\Delta$ in $G'$,  each vertex in $F$ is incident to an edge not in $G$. Thus, there must be an edge $e_1\in E(F) - E(G)$ as  $|V(F)|>|E(H, F)|$.  Let the two ends of $e_1$ be $x$ and $y$. 
 By the Cases 1 and 2, we may assume that there exists a vertex $v\in V(H)$ such that $d_{G'}(v) \le k-3$ or there exist two distinct vertices $v, v'\in V(H)\cap U$.  Let $G'' = G' -e_1 +e_2 +e_3$, where
 in the former case $e_2,e_3$ are edges between $v$ and $x,y$, respectively;  and  in the later case,  $e_2\in E(v, x)$ and $e_3\in E(v', y)$.   Clearly, $|E(G'')| = |E(G')| +1$ and $\D(G'') \le k-1$. 
 Thus we have $\chi'(G'')=k+1$ by the maximality of $|E(G')|$. 
 
Now by the Goldberg-Seymour conjecture,  there exists a subgraph $H'$ of $G''$ such that $|V(H')|$ is odd and $2|E(H')|>k(|V(H')|-1)$. If  $|\{e_2, e_3\}\cap E(H')| =1$, assume without loss of generality, $e_2\in E(H')$ and $e_3\notin 
 E(H')$. Since $\rho(G')\leq k$, we have that $H'-e_2$ is a $k$-dense subgraph of $G'$. Since $x\notin H$, $x\in H'$ and $v\in H'$, we see that $H$ and another maximal $k$-dense subgraph of $G'$ containing $H'-e_2$ having nonempty intersection, 
 giving a contradiction to   Lemma~\ref{density1}.   If $e_2, e_3\in E(H')$, we similarly see that $H$ and another maximal $k$-dense subgraph of $G'$ containing $H'-e_2-e_3+e_1$ having nonempty intersection, giving a contradiction to
 Lemma~\ref{density1}. Since all possibilities reach contradictions, we must have $H=G'$ and $G'$ is $k$-dense.

\flushleft{\bf Case 4.}  $U =\{u\}$ is a singleton and $d_{G'}(u) \leq k-3$.

Let $F= G' -u$. We first assume that  $E(F) - E(G)=\emptyset$. Since every vertex $v$ in $V(G')$ other than $u$  has degree $k-1>\Delta(G)$, $v$ must have been joined by an edge in $E(G')-E(G)$ to $u$. Thus $d_{G'}(u)\geq n-1$. By counting the total degree, we have the following:
\[
2|E(G')|=(k-1)(n-1)+d_{G'}(u)=k(n-1)+d_{G'}(u)-(n-1)\geq k(n-1)
\]
Since $\chi'(G') =k$, we have $2|E(G')| =k(n-1)$, so $G'$ is $k$-dense. 

We then assume that there is an edge $e\in E(F)-E(G)$. Let the two ends of $e$ be $x$ and $y$, and let $G'' = G'-e +e_2 +e_3$, where $e_2$ and $e_3$ are two edges between $u$ and $x$ and between $u$ and $y$, respectively.   Clearly, $|E(G'')| = |E(G')| +1$ and $\D(G'') = k-1$. Thus we have $\chi'(G'')=k+1$ by the maximality of $|E(G')|$. 
By the Goldberg-Seymour conjecture,  there exists a subgraph $H'$ of $G''$ such that $|V(H')|$ is odd and $2|E(H')|>k(|V(H')|-1)$. Similarly as before, we must have that $H'-e_2-e_3+e$ is $k$-dense subgraph of $G'$ containing $u$. Since $u$ is the only vertex in $U$, we are done by Case 3.


	\end{proof}

\bibliographystyle{amsplain}
\bibliography{total}

\providecommand{\bysame}{\leavevmode\hbox to3em{\hrulefill}\thinspace}
\providecommand{\MR}{\relax\ifhmode\unskip\space\fi MR }
\providecommand{\MRhref}[2]{%
  \href{http://www.ams.org/mathscinet-getitem?mr=#1}{#2}
}
\providecommand{\href}[2]{#2}
\begin{thebibliography}{1}

\bibitem{Behzad}
M.~Behzad, \emph{G{raphs} {and} {Their} {Chromatic} {Numbers}}, ProQuest LLC,
  Ann Arbor, MI, 1965, Thesis (Ph.D.)--Michigan State University. \MR{2615420}

\bibitem{chen2019proof}
Guantao Chen, Guangming Jing, and Wenan Zang, \emph{Proof of the
  {G}oldberg-{S}eymour {C}onjecture on {E}dge-{C}olorings of {M}ultigraphs},
  \arxiv{1901.10316} (2019).

\bibitem{Goldberg}
M.~K. Goldberg, \emph{On multigraphs of almost maximal chromatic
  class\,(russian)}, Discret. Analiz. \textbf{23} (1973), 3--7.

\bibitem{goldberg84}
\bysame, \emph{Edge-coloring of multigraphs: recoloring technique}, J. Graph
  Theory \textbf{8} (1984), no.~1, 123--137. \MR{732026}

\bibitem{10.1007/BFb0070378}
R.~P. Gupta, \emph{On the chromatic index and the cover index of a multigraph},
  Theory and Applications of Graphs (Berlin, Heidelberg) (Yousef Alavi and
  Don~R. Lick, eds.), Springer Berlin Heidelberg, 1978, pp.~204--215.

\bibitem{Seymour}
P.~Seymour, \emph{On multicolourings of cubic graphs, and conjectures of
  {F}ulkerson and {T}utte}, Proc. London Math. Soc. (3) \textbf{38} (1979),
  no.~3, 423--460. \MR{532981 (81j:05061)}

\bibitem{StiebSTF-Book}
M.~Stiebitz, D.~Scheide, B.~Toft, and L.~M. Favrholdt, \emph{Graph edge
  coloring}, Wiley Series in Discrete Mathematics and Optimization, John Wiley
  \& Sons, Inc., Hoboken, NJ, 2012, Vizing's theorem and Goldberg's conjecture,
  With a preface by Stiebitz and Toft. \MR{2975974}

\bibitem{vizing68}
V.~G. Vizing, \emph{Some unsolved problems in graph theory}, Uspehi Mat. Nauk
  \textbf{23} (1968), no.~6 (144), 117--134. \MR{0240000}

\end{thebibliography}

\end{document}